\documentclass[12pt]{amsart} 
\usepackage{amsmath,amssymb,amsthm,mathrsfs}
\usepackage[english]{babel}
\usepackage[utf8]{inputenc}

\usepackage{color}
\usepackage{graphicx}
\def\r{\mathbb R}

\def\W{\mathsf W}
\def\T{\mathcal T}

\newtheorem{theorem}{Theorem}[section]
\newtheorem{proposition}[theorem]{Proposition}
\newtheorem{corollary}[theorem]{Corollary}

 \theoremstyle{definition}
 \newtheorem{definition}[theorem]{Definition}
\newtheorem{example}[theorem]{Example}
\newtheorem{remark}[theorem]{Remark}
 \newtheorem{convention}[theorem]{Convention}
\begin{document}

\title[Curves in $\mathbb{R}^{2}$ endowed with a canonical linear connection]{Curvatures of curves in $\mathbb{R}^{2}$ endowed with a canonical linear connection}
\author{Muhittin Evren Aydin}
\address{Department of Mathematics, Faculty of Science, Firat University. 23200 Elazig, Turkey.}
\email{meaydin@firat.edu.tr}
\address{Department of Geometry and Topology. University of Granada.   18071 Granada, Spain}
\email{rcamino@ugr.es}
\author{Rafael L\'opez}

\begin{abstract}
We investigate the geometry of plane curves in $\r^2$ endowed with a canonical linear connection. By introducing the notions of the tangential and geodesic curvatures, we prove the existence and uniqueness theorems for curves with prescribed curvatures. We find all curves with constant curvature and we generalize the curves whose curvatures are linear functions of the arc length parameter.  
\end{abstract}

\maketitle

\section{Introduction and main results}

The study of linear connections on Riemannian manifolds that deviate from the classical Levi-Civita connection has a rich history in differential geometry. A prominent role is played by semi-symmetric connections, introduced by Friedmann and Schouten \cite{fs} and later systematically developed by Yano \cite{ya}. 

Let $(M, \langle,\rangle)$ be a Riemannian manifold and let $\nabla$ be a linear connection on $M$. The torsion tensor $\T$ of $\nabla$ is defined by
$$\T(X,Y)=\nabla_{X}Y-\nabla_{Y}X-[X,Y], \quad X,Y\in\mathfrak{X}(M),$$
where $[\cdot,\cdot]$ is the Lie bracket of vector fields. We say that $\nabla$ is {\it symmetric} if $\T=0$ identically. Moreover, the connection  $\nabla$ is said to {\it semi-symmetric} if there is a nonzero vector field $\W\in\mathfrak{X}(M)$ such that
$$\T(X,Y)=\langle \W,Y\rangle X-\langle \W,X\rangle Y, $$
for all $  X,Y\in\mathfrak{X}(M)$. If, in addition, $\nabla\langle,\rangle=0$ (respectively, $\nabla\langle,\rangle\ne0$) then $\nabla$ is called a {\it semi-symmetric metric connection} (resp, {\it semi-symmetric non-metric connection}). 

The notion of a semi-symmetric metric connection was introduced by  Hayden \cite{hay} whereas that of a non-metric  was introduced by Agashe and Chafle \cite{ac}. Over the last decades, these connections have been deeply investigated in relation to hypersurfaces, submanifolds, and curvature inequalities (see, e.g., \cite{im, lyl,mc1,na}).

We assume that the ambient space $M$ is the Euclidean plane $\r^{2}$, equipped with the Euclidean metric $\langle,\rangle$ and let  $\W\in\mathfrak{X}(\r^2)$. Let $m, n$ be smooth functions on $\r^{2}$, not both identically zero. In this paper, we consider the linear connection
\begin{equation}\label{e1}
\nabla_{X}Y=\nabla_{X}^{0}Y+m\langle \W,Y\rangle X+n\langle X,Y\rangle \W, \quad X,Y\in\mathfrak{X}(\r^{2}),
\end{equation}
where $\nabla^{0}$ denotes the Levi-Civita connection on $\r^{2}$ associated with the metric $\langle,\rangle$. Recall that the torsion tensor $\T^{0}$ of $\nabla^{0}$ satisfies $\T^{0}=0$ and $\nabla^{0}\langle,\rangle=0$. Moreover, if $m=n=0$, then $\nabla=\nabla^0$.

The connection given by \eqref{e1} is a special case of the connection considered in \cite[Equation (2.4)]{trip}, which corresponds  to the choices 
$$\varphi_1=-m\mathrm{Id}, \quad
\varphi=0, \quad
U=-U_2=-\W, \quad
f_1=0, \qquad
f_2=-(m+n).
$$

A direct calculation shows that the linear connection $\nabla$ given by \eqref{e1} and its torsion tensor $\T$ satisfy the following relations:
\begin{equation}\label{e2}
\begin{aligned}
\T(X,Y) &= m(\langle \W,Y\rangle X-\langle \W,X\rangle Y), \\
(\nabla_{X}\langle,\rangle)(Y,Z) &= -(m+n)(\langle X,Y\rangle\langle \W,Z\rangle+\langle X,Z\rangle\langle \W,Y\rangle),
\end{aligned}
\end{equation}
for all $X, Y, Z\in\mathfrak{X}(M)$. Hence, the connection $\nabla$ given by \eqref{e1} is
\begin{enumerate} 
\item[(i)] symmetric (resp. semi-symmetric) if $m$ is identically zero (resp. nowhere zero) on $\r^2$;
\item[(ii)] metric (resp. non-metric) if $m+n$  is identically zero (resp. nowhere zero) on $\r^2$.
\end{enumerate}

When we state that a function on $\r^2$ is nowhere zero, we mean that there is no   open subset of $\r^2$ on which the function vanishes identically. Thus, when dealing with the non-metric the case, the functions $m$ and $n$ must be chosen appropriately. For example, let $(x,y)$ be the canonical coordinates on $\r^2$ and suppose that $n=x+y-m.$ Then $m+n=x+y,$
which vanishes along the straight line $y=-x$. Therefore, the choice $m+n$ fails to satisfy our assumption.

While much of the literature focuses on curvature tensors of the ambient space or general submanifold theory, the intrinsic geometry of curves under these non-standard connections remains largely unexplored, especially when the   vector field $\W$ possesses a canonical nature. In this paper, we bridge this gap by investigating the geometry of curves in the Euclidean plane $\r^2$ endowed with a   linear connection $\nabla$, where the vector field $\W$ is a constant unit vector field on $\r^2$. The presence of this vector field, together with   two functions $m$ and $n$, breaks the isotropic nature of the Euclidean plane and induces a non-trivial spatial distortion.  

To the best of the authors' knowledge, the only work  explicitly considering curves with respect to semi-symmetric connections is \cite{gu}. In that paper, the author studied geodesics with respect to a semi-symmetric metric connection in three-dimensional Euclidean space, as well as in three-dimensional Kenmotsu and Sasakian manifolds.

The choice of the linear connection $\nabla$ given by \eqref{e1} depends on the selection of the vector field $\W$. One of the simplest examples of such connections occurs when $\W$ is one of the two canonical vector fields $\partial_{x}$ or $\partial_{y}$.  More generally, in this paper, we will consider that $\W$ is a   constant unit vector field and thus, there is $\theta_0\in\r$ such that 
$$\W=\cos\theta_0\, \partial_x+\sin\theta_0\, \partial_y.$$

\begin{definition}
A linear connection $\nabla$ on $\r^2$ of the form \eqref{e1} is called a canonical linear connection if $\W$ is a constant unit vector field. In short, such a connection will be called a cl-connection.
\end{definition}

The geometric properties of $\nabla$ are therefore determined by the vector field $\W$ together with the functions $m$ and $n$. In particular, the torsion tensor $\T$ and the covariant derivative of the metric depend  on these functions. Consequently, the connection $\nabla$ is completely determined by the triplet $(\W,m,n)$.

\begin{convention}
Unless otherwise stated, the expression "the cl-connection $\nabla$ determined by $(\W,m,n)$" refers to the linear connection $\nabla$ defined by \eqref{e1}, where $\W$ is a   constant unit vector field on $\r^2$.
\end{convention}

Let $\gamma:I\subset\r\rightarrow\r^{2}$, $\gamma=\gamma(s)$, be a smooth curve parametrized by arc length. Denote by $T$ and $N$ the unit tangent and normal vector fields along $\gamma$, respectively. Along this paper,  $\kappa$ denotes the Frenet curvature of $\gamma$. The Frenet equations are given by
$$\begin{aligned}
\nabla_{T}^{0}T &=\kappa N \\
\nabla_{T}^{0}N &=-\kappa T,
\end{aligned}$$
where $\nabla^{0}$ is the Levi-Civita connection on $\r^{2}$. Using \eqref{e1}, we then obtain:
\begin{equation*}
\begin{aligned}
 \nabla_TT&=(m+n)\langle \W,T\rangle T +(\kappa+n\langle \W,N\rangle )N,\\
\nabla_{T}N &=(m\langle \W,N\rangle-\kappa)T.
\end{aligned}
\end{equation*}

To motivate the new curvatures  for the connection $\nabla$, recall that in the standard setting, the geodesics of $\r^{2}$ are straight lines. If $\gamma$ is a geodesic, then its Frenet curvature $\kappa$ is identically zero. In fact, when $\gamma$ is not necessarily parametrized by arc length, straight lines are characterized not only by the vanishing of their Frenet curvature but also by the following two additional properties:
\begin{enumerate}
\item $\nabla_{\gamma'}^{0}\gamma'$ is linearly dependent on $\gamma'$;
\item the unit normal vector field $N$ along $\gamma$ is parallel, that is,  $\nabla_{\gamma'}^{0}N=0$.
\end{enumerate}
Using these facts, we combine the normal acceleration component of $\nabla_T T$ and the tangential  component of $\nabla_T N$ into a single quadratic scalar. This motivates the following definition of  geodesic curvature. Meanwhile, we also define the tangential curvature as the scalar part of the tangential component of $\nabla_TT$.

\begin{definition}\label{def12}
Let $\nabla$ be a cl-connection on $\r^{2}$ determined by $(\W,m,n)$ and let $\gamma\colon I\to\r^{2}$ be a smooth curve parametrized by arc length. With respect to  $\nabla$, 
\begin{enumerate}
\item The tangential curvature $\kappa_t$ of $\gamma$ is defined by
\begin{equation}\label{d1}
\kappa_{t}=(m+n)\langle \W,T\rangle.
\end{equation}
\item   The geodesic curvature $\kappa_{g}\geq 0$ of $\gamma$   is defined by
\begin{equation}\label{d2}
\kappa_{g}^{2}=( \kappa+n\langle \W,N\rangle)^{2}+(m\langle \W,N\rangle-\kappa)^{2}.
\end{equation}
In addition, if $\kappa_{g}=0$ along $\gamma$, then   $\gamma$ is said to be a geodesic with respect to the connection $\nabla$.

\end{enumerate}
\end{definition}

With these definitions, a geodesic curve $\gamma$ with respect to a cl-connection $\nabla$ satisfies that $\nabla_{T}T$ is parallel to $ T$ and  that  $\nabla_{T}N=0$, where $T=\gamma'$ and $N$ are the unit tangent and normal vector fields along $\gamma$, respectively.  

The main purpose of this paper is   to establish  the fundamental theorem of curve theory: the existence and uniqueness of 
curves with prescribed (tangential or geodesic) curvatures. The  first main result   solves this problem for a 
prescribed tangential curvature via  quadratures (see Theorem \ref{t32}). Our second main result addresses the 
prescribed geodesic curvature problem. Because the geodesic curvature $\kappa_g$ involves the Frenet curvature along with the orthogonal projection of the   vector field $\W$, prescribing $\kappa_g$ translates into a more complex problem involving  
differential equations. Nevertheless, we prove that given a prescribed function $k(s) > 0$, there is a unique 
curve whose geodesic curvature is $k(s)$, which  is   governed by a first-order non-linear system of ordinary 
differential equations for the tangent angle (see Theorem \ref{t42}). More precisely, for any given initial position and direction, 
there exists a unique curve realizing this geodesic curvature, provided a strict algebraic discriminant condition holds at the 
initial point (see Theorem \ref{t44}).

The present paper is organized as follows. In Section \ref{s2}, we present  the behavior of the generalized curvatures under geometric transformations, along with the study of their total integrals for simple closed curves. 
Section \ref{s3} is devoted to the analysis of plane curves with a prescribed tangential curvature $\kappa_t$, determining those curves where $\kappa_t=0$ identically, proving the theorem of existence and uniqueness for a prescribed tangential curvature, and characterizing those curves with constant tangential curvature. In Section \ref{s4}, the development is analogous for the geodesic curvature. We characterize the geodesics  of the connection,  establish the existence and uniqueness theorems for curves with prescribed geodesic curvature and finally, classify those curves with constant geodesic curvature.  We also offer a geometric interpretation of these curvatures via the relative angle function (see Theorem \ref{t46}). Finally, in Section \ref{s5}, we construct   explicit parametrizations for the modified $\nabla$-clothoids and logarithmic $\nabla$-spirals.

 \section{Preliminaries}\label{s2}

For the connection given by \eqref{e1}, we assume that the connection functions $m$ and $n$  are not both identically zero, as    otherwise it reduces to the Levi-Civita connection. A first question arises as to  what happens with the notions of   tangential and geodesic curvatures when one of the connection functions vanishes. The following result is immediate from Definition  
\ref{def12}.

\begin{proposition} 
Let $\gamma(s)\subset \r^2$ be a smooth curve parametrized by arc length. Let $\kappa_t$ and $\kappa_g$ denote the tangential curvature and the geodesic curvature of $\gamma$ with respect to the cl-connection determined by $(\W,m,n)$, respectively.
\begin{enumerate}
    \item[(i)] If $m = 0$, then  
    \begin{align*}
    \kappa_t &= n\langle \W, T \rangle, \\
    \kappa_g &= \sqrt{(\kappa + n\langle \W, N \rangle)^2 + \kappa^2}.
    \end{align*}
    \item[(ii)]  If $n = 0$, then 
    \begin{align*}
    \kappa_t &= m\langle \W, T \rangle, \\
    \kappa_g &= \sqrt{\kappa^2 + (m\langle \W, N \rangle - \kappa)^2}.
    \end{align*}
\end{enumerate}

\end{proposition}
 
Next, we investigate how the geometry induced by a cl-connection behaves under transformations of $\r^2$. Since $\W$ is 
a constant vector field and the Frenet curvature $\kappa$ is translation-invariant, it is straightforward to see that translations of $\r^2$ do not affect  $\kappa_t$ or $\kappa_g$. However, this does not hold for dilations. In 
classical Euclidean geometry, scaling a curve by a factor $c > 0$ simply scales its Frenet curvature $\kappa$ by $1/c$. However, the presence of the vector field $\W$ and the functions $m$ and $n$ introduces a   scale distortion.

\begin{proposition} 
Let $\nabla$ be a cl-connection on $\r^2$ determined by $(\W,m,n)$. Let $\gamma(s)$ be a smooth curve parametrized by arc length and let $\tilde{\gamma}(\tilde{s}) = c\gamma(s)$ be its homothetic transformation by a scale factor $c > 0$, where $\tilde{s} = cs$ is the arc length of $\tilde{\gamma}$. Then, the tangential and geodesic curvatures of $\tilde{\gamma}$ with respect to $\nabla$ are given by:
\begin{equation} \label{eh}
\begin{split}
\tilde{\kappa}_t(\tilde{s}) &= \left( \tilde{m}+ \tilde{n} \right) \langle \W, T(s) \rangle, \\
\tilde{\kappa}_g^2(\tilde{s}) &= \left( \frac{\kappa(s)}{c} + \tilde{n}\langle \W, N(s) \rangle \right)^2 + \left( \tilde{m}\langle \W, N(s) \rangle - \frac{\kappa(s)}{c} \right)^2,
\end{split}
\end{equation}
where $\tilde{m}(\tilde{s}) = m(c\gamma(s))$ and $\tilde{n}(\tilde{s}) = n(c\gamma(s))$.
\end{proposition}

\begin{proof}
The proof is a direct computation using the definitions of the tangential curvature and geodesic curvature of $\tilde{\gamma}$ given in \eqref{d1} and \eqref{d2}. Let $\tilde{\gamma}(\tilde{s}) = c\gamma(\tilde{s}/c)$ be the scaled curve. Taking the derivative with respect to the new arc length parameter $\tilde{s}$, the unit tangent vector is 
$$\tilde{T}(\tilde{s}) = \frac{d\tilde{\gamma}}{d\tilde{s}} = c \frac{d\gamma}{ds} \frac{ds}{d\tilde{s}} = T(s).$$
Since the tangent vector remains invariant, so does the normal vector, meaning $\tilde{N}(\tilde{s}) = N(s)$. Furthermore, taking the second derivative yields 
$$\frac{d\tilde{T}}{d\tilde{s}} = \frac{dT}{ds} \frac{ds}{d\tilde{s}} = \frac{\kappa(s)}{c} N(s),$$
which recovers the classical scaling law for Euclidean curves: $\tilde{\kappa}(\tilde{s}) = \frac{\kappa(s)}{c}$. Moreover,  $\langle \W, \widetilde{T} \rangle = \langle \W, T \rangle$ and $\langle \W, \widetilde{N} \rangle = \langle \W, N \rangle$. 

However, the connection functions $m$ and $n$ depend on the spatial position and must be evaluated at the new coordinates $\tilde{\gamma}(\tilde{s}) = c\gamma(s)$. Substituting these scaled elements $\tilde{m}$, $\tilde{n}$, $\widetilde{T}$, $\widetilde{N}$, and $\tilde{\kappa}$ directly into Definition \ref{def12} yields the  expressions in \eqref{eh}.  
\end{proof}

In the classical theory of planar curves, the integral of the Frenet curvature along a simple closed curve is a topological invariant, evaluating to $\pm 2\pi$ by the well-known Hopf's Umlaufsatz (see \cite[Theorem 3.1.4]{pre}). It is a natural question to ask whether the total generalized tangential curvature exhibits a similar behavior.  

\begin{proposition} \label{totcu}
Let $\gamma(s)\subset \r^2$ be a smooth simple closed curve parametrized by arc length, bounding a compact domain $\Omega\subset \r^2$. Let $\nabla$ be a cl-connection on $\r^2$ determined by $(\W,m,n)$, where $\W = (\cos\theta_0, \sin\theta_0)$ for some $\theta_{0}\in\r$. Then, the total tangential curvature of $\gamma$ with respect to $\nabla$  is given by
\begin{equation}\label{k3}
 \oint_\gamma \kappa_t(s) \, ds = \iint_\Omega \left( \sin\theta_0 \frac{\partial (m+n)}{\partial x} - \cos\theta_0 \frac{\partial (m+n)}{\partial y} \right) dx dy. 
 \end{equation}
In particular, if   $m+n = \lambda \in \r$ is constant, then   $\oint_\gamma \kappa_t(s) \, ds = 0$.
\end{proposition}

\begin{proof}
From \eqref{d1}, we know  that $\kappa_t(s) = (m+n)\langle \W, T(s) \rangle=(m+n)\langle \W, \gamma'(s) \rangle$. On $\r^2$, define the vector field $$V(x,y) = (m(x,y)+n(x,y))\W.$$
 Writing $V$ in coordinates, we have $V(x,y) = (P(x,y), Q(x,y))$ where
$$ P(x,y) = (m+n)\cos\theta_0, \quad Q(x,y) = (m+n)\sin\theta_0. $$
Applying Green's Theorem to  $\gamma$, we obtain
$$ \oint_\gamma \langle V, T \rangle \, ds = \oint_\gamma P \, dx + Q \, dy = \iint_\Omega \left( \frac{\partial Q}{\partial x} - \frac{\partial P}{\partial y} \right) dx dy. $$
Substituting the partial derivatives of $P$ and $Q$ yields \eqref{k3}.   The last   statement is immediate.  
\end{proof}

\begin{remark}
Unlike the total Frenet curvature, the total tangential curvature is influenced by the functions $m$ and $n$, as shown in Proposition \ref{totcu}. Thus, its value is not uniquely determined by the topology of the curve. However, the total tangential curvature vanihes identically   whenever $m+n$ is constant. In particular, this is the case for metric cl-connections.
\end{remark}

In   contrast to the tangential curvature, the total geodesic curvature $\kappa_g$ along a closed curve is more difficult  due to its   non-linear, quadratic definition (see \eqref{d2}). However, we can establish a lower bound for the total geodesic curvature of simple closed curves.

\begin{proposition} 
Let $\gamma $ be a smooth simple closed curve in $\r^2$ parametrized by arc length $s$. Let $\nabla$ be a cl-connection on $\r^2$ determined by $(\W,m,n)$. Then, the total geodesic curvature with respect to $\nabla$ satisfies the inequality
\begin{equation}\label{eu}
\oint_\gamma \kappa_g(s) \, ds \ge \frac{1}{\sqrt{2}} \left| 4\pi + \oint_\gamma (n - m)\langle \W, N \rangle \, ds \right|.
\end{equation}
If $m=n$,   we have 
\begin{equation}\label{eu2}
\oint_\gamma \kappa_g(s) \, ds \ge 2\sqrt{2}\pi.
\end{equation}
\end{proposition}

\begin{proof}
From \eqref{d2}, we write   $\kappa_g = \sqrt{A^2 + B^2}$, where $A = \kappa + n\langle \W, N \rangle$ and $B = m\langle \W, N \rangle - \kappa$. Utilizing the  inequality $\sqrt{A^2 + B^2} \ge \frac{1}{\sqrt{2}}|A - B|$ for any real numbers $A$ and $B$, we have
$$ A - B = (\kappa + n\langle \W, N \rangle) - (m\langle \W, N \rangle - \kappa) = 2\kappa + (n - m)\langle \W, N \rangle. $$
This allows us to bound $\kappa_g$   by
$$ \kappa_g(s) \ge \frac{1}{\sqrt{2}} \left| 2\kappa(s) + (n - m)\langle \W, N \rangle \right|. $$
Integrating both sides along   $\gamma$, we obtain
$$ \oint_\gamma \kappa_g(s) \, ds \ge \frac{1}{\sqrt{2}} \left| 2\oint_\gamma \kappa(s) \, ds + \oint_\gamma (n - m)\langle \W, N \rangle \, ds \right|. $$
Since $\gamma$ is a simple closed curve, Hopf's Umlaufsatz states that   $\oint_\gamma \kappa(s) \, ds = 2\pi$ (assuming a counterclockwise orientation). Substituting this     into the above  inequality   yields \eqref{eu}. If $m = n$, then we arrive at \eqref{eu2},  
which completes the proof.
\end{proof}

\begin{remark}
The bound given in \eqref{eu2} is sharp. In the case $m=n=0$ ($\nabla=\nabla^0$), we have $A=-B=\kappa$, which means $\kappa_g = \sqrt{2}|\kappa|$. For any convex simple closed curve where $\kappa > 0$, the integral yields   $\oint_\gamma \kappa_g \, ds =   2\sqrt{2}\pi$. 
\end{remark}

\section{Tangential curvature of plane curves}\label{s3}

In this section, we investigate the tangential curvature of plane curves with respect to a cl-connection $\nabla$. Recall that the tangential curvature $\kappa_t$ vanishes identically whenever $\nabla$ is a metric cl-connection. The following result characterizes the curves for which $\kappa_t$ vanishes identically in the non-metric case.

\begin{proposition}
Let $\nabla$ be a non-metric cl-connection on $\r^2$ determined by $(\W,m,n)$, where $\W = (\cos\theta_{0},\sin\theta_{0})$ for some $\theta_{0}\in\r$. Let $\gamma$ be a smooth curve in $\r^2$ parametrized by arc length. The tangential curvature $\kappa_{t}$ of $\gamma$ with respect to $\nabla$ vanishes identically if and only if
$$
\gamma(s)=(-\sin\theta_{0},\cos\theta_{0})s+\vec{w}, \quad \vec{w}\in \r^2.
$$
\end{proposition}

\begin{proof}
By the definition of $\kappa_{t}$ in \eqref{d1}, we have $(m+n)\langle \W,T\rangle=0$, where $T$ is the unit tangent vector field of $\gamma$. Since $\nabla$ is non-metric, we have $m+n\neq 0$ along $\gamma$, which implies $\langle \W,T\rangle=0$. Since $\W$ is constant, it follows that $T$ is constant and thus $\gamma$ is a straight line of the stated form. The converse follows directly from \eqref{d1}.
\end{proof}
  
When $\nabla$ is non-metric, not every straight line has vanishing tangential curvature. In particular, straight lines normal to $\W$ satisfy $\kappa_t=0$.

Let $R_{\theta}$ denote a rotation about the origin $(0,0)\in\r^2$ by an angle $\theta$. The following result establishes the existence and uniqueness of plane curves whose tangential curvature $\kappa_{t}$ is prescribed.

\begin{theorem}\label{t32}
Let $\nabla$ be a non-metric cl-connection on $\r^{2}$ determined by $(\W,m,n)$, where $\W = (\cos\theta_{0},\sin\theta_{0})$ for some $\theta_{0}\in\r$. Let also $k:I\rightarrow\r$ be a nowhere zero smooth function and $|k|\leq |m+n|$. Then, there exists a smooth curve $\gamma:I\rightarrow\r^2$, parametrized by arc length 
such that the tangential curvature of $\gamma$ with respect to $\nabla$ is $k$. In addition, up to translations of 
$\r^2$, the curve $\gamma$ is uniquely given by
\begin{equation}\label{e4}
\gamma(s)=R_{\theta_{0}}\circ\left(\pm\int^{s}f(t)dt,\pm\int^{s}\sqrt{1-f(t)^{2}}dt\right),
\end{equation}
where $f(s):=\frac{k(s)}{(m+n)(\gamma(s))}$.
\end{theorem}

\begin{proof}
 By the definition of $\kappa_t$ in \eqref{d1},  we get
\begin{equation}\label{e5}
\langle \W,T\rangle=\frac{k}{m+n}=f.
\end{equation}
  Note that $|f|\leq 1$ along $\gamma$, because $\W$ and $T$ are unit vector fields and $k$ is nowhere zero. Taking the derivative of \eqref{e5} with respect to the Levi-Civita connection and applying the Frenet formulas, we find
\begin{equation}\label{e52}
\kappa\langle \W,N\rangle=f',
\end{equation}
where $\kappa$ is the Frenet curvature of $\gamma$. 

Assume that $\gamma$ is not a straight line. Then, $f$ is a non-constant function. Notice that the proof remains true when $\gamma$ is a straight line. Since $\W$ is a unit vector field, we have 
$\langle \W,T\rangle^{2}+\langle \W,N\rangle^{2}=1$ which, in combination of \eqref{e5} and \eqref{e52}, leads to 
$$\kappa=\pm\frac{f^{\prime}}{\sqrt{1-f^{2}}}.$$
Integrating, we derive
\begin{equation}\label{e6}
\begin{aligned}
\int^{s}\kappa(u)\, du=\begin{cases}\pm \sin^{-1}(f) +\theta,\\ 
 \mp \cos^{-1}(f)+\theta\end{cases}
\end{aligned}
\end{equation}
where $\theta\in\r$ is a constant. Recall that $\gamma$ is uniquely determined by its Frenet curvature $\kappa$ via  the expression
\begin{equation}\label{e7}
\gamma(s)=\left(\int^{s}\cos\left(\int^{t}\kappa(u)\, du\right)dt,\int^{s}\sin\left(\int^{t}\kappa(u)\, du\right)dt\right).
\end{equation}
By taking the inverse cosine in \eqref{e6} and substituting it into \eqref{e7}, we obtain the parametrization given by \eqref{e4}.
On the other hand, to determine the constant $\theta$, we analyze the alignment of the tangent vector. Let $\alpha(s) = 
\int^s \kappa(u)du$ be the angle function from \eqref{e7}, so that $T(s) = (\cos\alpha(s), \sin\alpha(s))$. Computing the 
inner product yields $\langle \W, T(s) \rangle = \cos(\alpha(s) - \theta_0) = f(s)$. Substituting the angular relation from \eqref{e6}
 into this  identity and expanding it via standard trigonometric formulas for the sum and difference of angles, we obtain
$$(\cos\theta_{0}\cos\theta+\sin\theta_{0}\sin\theta-1)f(s)\pm(\sin\theta_{0}\cos\theta-\cos\theta_{0}\sin\theta)\sqrt{1-f(s)^{2}}=0,$$
for all $s$, or equivalently, 
$$(\cos(\theta_{0}-\theta) -1)f(s)\pm\sin(\theta_{0}-\theta)\sqrt{1-f(s)^{2}}=0.$$
Since $f$ is non-constant, the functions $f(s)$ and $\sqrt{1-f(s)^{2}}$ are linearly independent over $\mathbb{R}$. Therefore, their coefficients must vanish simultaneously, which  implies $\cos(\theta_{0}-\theta) = 1$, and consequently $\theta=\theta_{0}$. Substituting this values  into the   system \eqref{e7}   yields the parametrization \eqref{e4}.
  
Conversely, differentiating \eqref{e4} yields the unit tangent vector 
$$T(s) = R_{\theta_0}\circ(f(s), \pm\sqrt{1-f(s)^2}).$$
 Since  $\W = R_{\theta_0}\circ(1,0)$, their inner product reduces to $\langle \W, T(s) \rangle = f(s)$,
 which   verifies that $\kappa_t = (m+n)\langle \W, T \rangle = k$.
 \end{proof}

As a consequence of Theorem~\ref{t32}, we classify the curves with   constant tangential curvature.

\begin{corollary}\label{co34}
Let $\nabla$ be a non-metric cl-connection determined by $(\W,m,n)$, where $\W = (\cos\theta_{0},\sin\theta_{0})$ for  any $\theta_0 \in  \r$. Let $c \in \r \setminus \{0\}$ be a prescribed constant. If $\gamma:I\rightarrow\r^2$ is a smooth curve parametrized by arc length with constant tangential curvature $\kappa_{t} = c$, then up to rigid translations, $\gamma$ is uniquely given by
\begin{equation}\label{e10}
\gamma(s)=R_{\theta_{0}}\circ\left(\pm \int^{s}\frac{c}{m+n}\,dt, \, \pm\int^{s}\sqrt{1-\frac{c^2}{(m+n)^2}}\,dt\right),
\end{equation}
where $|m+n| \geq |c|$ along the trace of the curve. Furthermore, $\gamma$ is a straight line if and only if $m+n$ is constant along $\gamma$.
\end{corollary}
\begin{proof}
By setting the prescribed curvature function $k(s) = c$ in Theorem~\ref{t32}, the  function $f$ reduces  to $f(s) = \frac{c}{m+n}$.
 Substituting this  into \eqref{e4}  yields the parametric form \eqref{e10}. Since   
$|f(s)| \leq 1$, this implies the restriction $|m+n| \geq  |c|$. 

Finally, from equation \eqref{e52}, the classical Frenet curvature satisfies $\kappa \langle \W, N \rangle = f'$. If $m+n$ is constant 
along the curve, then $f' = 0$, which forces $\kappa = 0$ (since $\langle \W, N \rangle \neq 0$), confirming that 
$\gamma$ is a straight line. Conversely, if $\gamma$ is a straight line, then $\kappa = 0$, which implies $f' = 0$, 
meaning that $m+n$ must remain constant along the trace of $\gamma$.
\end{proof}

To illustrate this result,   we present an    example.

\begin{example}\label{ex34}
Consider the non-metric cl-connection $\nabla$   determined by $(\W,m,n)$, where $\W = \partial_x = (1,0)$ (hence $\theta_0 = 0$) and the functions $m,n$ satisfy $m(x,y) + n(x,y) =e^x$. Let us prescribe a constant tangential curvature $\kappa_t(s) = 1$ for $s \in [1, \infty)$. We aim to find the   curve $\gamma$ satisfying $\kappa_t(s) = 1$. According to Corollary \ref{co34}, we set 
$$f(s) = \frac{1}{m+n} = \frac{1}{s}$$ 
along the curve, which implies $(m+n)\circ\gamma(s) = s$. Notice that $|(m+n)\circ\gamma|\geq \kappa_t$ for all $s$. Utilizing  \eqref{e7}, we find the coordinates of $\gamma$:
\begin{equation*}
\begin{split}
 x(s) &= \int^s f(t) \, dt = \int^s \frac{1}{t} \, dt = \log s, \\
 y(s) &= \int^s \sqrt{1 - f(t)^2} \, dt = \int^s \sqrt{1 - \frac{1}{t^2}} \, dt = \sqrt{s^2 - 1} - \arctan\left(\sqrt{s^2 - 1}\right). 
 \end{split}
 \end{equation*}
Thus, up to an initial translation, the curve is  expressed as (see Fig. \ref{fig1}.)
\begin{equation}\label{constex}
 \gamma(s) = \left( \log s, \, \sqrt{s^2 - 1} - \arctan\left(\sqrt{s^2 - 1}\right) \right), \quad s \in [1, \infty).  \end{equation}
\end{example}

\begin{figure}[ht]
\centering
\includegraphics[width=0.3\textwidth]{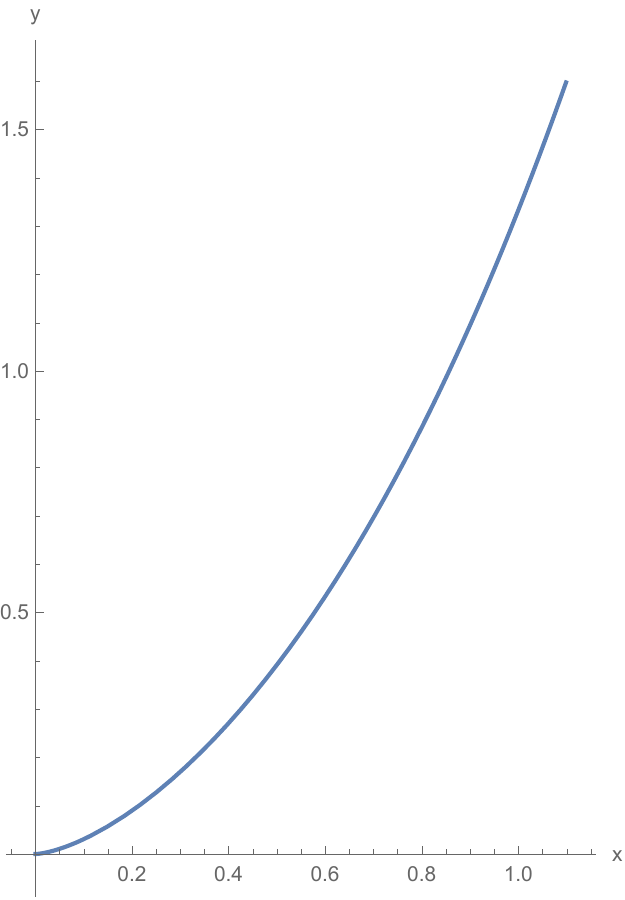}
\caption{The curve defined by \eqref{constex} with constant tangential curvature $\kappa_t(s)=1$ for $s\in[1,3]$.}
\label{fig1}
\end{figure}

\begin{remark} 
Example~\ref{ex34} demonstrates the  lack of rotational invariance for the tangential curvature in contrast to the Euclidean geometry. If we rotate the curve $\gamma(s)$ by an 
angle of $\pi/4$  around the origin, we obtain   $\tilde{\gamma}(s) = R_{\pi/4}(\gamma(s)) = (\tilde{x}(s), \tilde{y}(s))$, where its 
new $x$-coordinate is  $\tilde{x}(s) = \frac{\sqrt{2}}{2}(x(s) - y(s))$. Since the connection is 
fixed,   the tangential curvature of the rotated curve at its new spatial coordinates is 
$$ \tilde{\kappa}_t(s) = e^{\tilde{x}(s)} \langle \W, \tilde{T}(s) \rangle = e^{\tilde{x}(s)}\tilde{x}'(s) = \frac{d}{ds} \left( e^{\tilde{x}(s)} \right). $$
For $\tilde{\kappa}_t(s)$ to remain identically $1$, we would need $e^{\tilde{x}(s)} = s + c$ for some constant $c$, which implies $\tilde{x}(s) = \log(s+c)$. However,  we have
$$ \tilde{x}(s) = \frac{\sqrt{2}}{2}\left( \log s - \sqrt{s^2 - 1} + \arctan\left(\sqrt{s^2 - 1}\right) \right), $$
which shows that $\tilde{x}(s) \not= \log(s+c)$. 
\end{remark}

It is natural to ask under what conditions the new   tangential curvature $\kappa_t$ coincides with  the Frenet curvature $\kappa$ of a plane curve. 
The following result addresses the question whether coincide along the curve.

\begin{corollary} Under the same hypothesis of Theorem \ref{t32},  we have  $\kappa_t=\kappa$   if and only if 
\begin{equation}\label{ii}
(m+n)\circ\gamma(s) = \frac{\kappa(s)}{\langle \W, T(s) \rangle},
\end{equation}
for all $s \in I$ where $\langle \W, T(s) \rangle \neq 0$. Alternatively, if $\theta(s)$ denotes the angle between   $\W$ and the tangent vector $T(s)$, this condition holds if and only if
\begin{equation} \label{ii2}
(m+n)\circ\gamma(s) = \frac{\theta'(s)}{\cos\theta(s)}.
\end{equation}
\end{corollary}

\begin{proof}
 Identity \eqref{ii} is a consequence of setting $\kappa_t  = \kappa $ and the definition of $\kappa_t$ in \eqref{d1}. Moreover, if $\cos\theta(s)=\langle\W,T(s)\rangle$, we know that $\theta'=\kappa$, whence \eqref{ii} implies \eqref{ii2}. 
\end{proof}

We know that circles are the curves with non-zero constant curvature $\kappa$. In the following result, we ask when these circles have constant tangential curvature.

\begin{proposition}
Let $\nabla$ be a non-metric cl-connection determined by $(\W,m,n)$, where $\W = (\cos\theta_{0},\sin\theta_{0})$ for
 some $\theta_0 \in  \r$.  A circle $\gamma$ of radius $R>0$ centered at the origin has a constant tangential curvature $\kappa_t = c$ (with $c \neq 0$) if and only if the functions $m$ and $n$ satisfy the following relation along the trace of the circle:
\begin{equation}\label{e11}
m(x,y) + n(x,y) = \frac{cR}{\pm\sqrt{R^2 - (x\cos\theta_0 + y\sin\theta_0)^2}},
\end{equation}
where $|x\cos\theta_0 + y\sin\theta_0| < R$.
\end{proposition}

\begin{proof}
Let $\gamma(s)$ be a circle of radius $R$ which, after a translation, can be parametrized by   $\gamma(s) = R ( \cos(s/R),  \sin(s/R) )$. Define   $u(x,y) = x\cos\theta_0 + y\sin\theta_0$. Evaluating this function along the trace of the curve $\gamma(s)$ yields
$$ u(s) =   R\cos\left(\frac{s}{R} - \theta_0\right). $$
Moreover, $ \langle \W, T(s) \rangle  = -\sin\left(\frac{s}{R} - \theta_0\right)$.  
Thus,   $\langle \W, T(s) \rangle = u'(s)$. This gives
$$ \langle \W, T(s) \rangle = \pm\sqrt{1 - \cos^2\left(\frac{s}{R} - \theta_0\right)} = \pm\sqrt{1 - \frac{u(s)^2}{R^2}}. $$
From the definition \eqref{d1} of $\kappa_t$, we have   $\kappa_t(s) = (m+n)\circ\gamma(s) \langle \W, T(s) \rangle$. Imposing the condition that $\kappa_t = c$ and substituting, we obtain
$$ c =(m+n)\circ\gamma(s) \left( \pm\frac{1}{R}\sqrt{R^2 -  u(s)^2 } \right). $$
Solving for the sum of the functions $m,n$ and reverting the curve parameter back to the ambient spatial coordinates $(x,y)$ via $u = x\cos\theta_0 + y\sin\theta_0$, we  arrive at \eqref{e11}. The strict inequality restriction ensures that the denominator is non-zero, completing the proof.
\end{proof}

\section{Geodesic curvature of plane curves}\label{s4}

In this section, we investigate the  geodesic curvature $\kappa_g$ induced by a cl-connection. As in the previous section, this study is   divided into three parts: 
 first, we classify the geodesics of the connection; second, we establish the existence result with a general prescribed geodesic 
 curvature function $k(s)$; and finally, we find the  curves with constant geodesic curvature. We begin by classifying the geodesics of 
 $\nabla$.

\begin{theorem} \label{geo-cl}
Let $\nabla$ be a cl-connection on $\mathbb{R}^{2}$ determined by $(\W,m,n)$. A smooth curve $\gamma$ parametrized by arc length is a geodesic with respect to $\nabla$ if and only if one of the following conditions holds:
\begin{enumerate}
    \item[(i)] $\gamma$ is a straight line parallel to the vector field $\W$.
    \item[(ii)] $\nabla$ is a metric connection ($m + n = 0$) and  the Frenet curvature of $\gamma$ holds $\kappa=m\langle \W, N \rangle$.
\end{enumerate}
\end{theorem}

\begin{proof}
From \eqref{d2},   $\gamma$ is a geodesic if and only  if
\begin{align}
\kappa + n\langle \W, N \rangle &= 0, \label{geo1} \\
m\langle \W, N \rangle - \kappa &= 0. \label{geo2}
\end{align}
Adding equations \eqref{geo1} and \eqref{geo2} yields  
\begin{equation}\label{gc}
(m + n)\langle \W, N \rangle = 0.
\end{equation}
\begin{enumerate}
    \item Suppose $\langle \W, N \rangle = 0$ along an interval. Then,   $\langle \W, T \rangle = \pm 1$. Since $\W$ is  constant, the tangent direction of $\gamma$ is fixed, meaning $\gamma$ is a straight line parallel to $\W$. Substituting $\langle \W, N \rangle = 0$ into either \eqref{geo1} or \eqref{geo2}  yields $\kappa = 0$. This gives that   $\gamma$ is a straight line.
    
    \item Suppose $\langle \W, N \rangle \neq 0$ along $\gamma$. Then, from  \eqref{gc}, we must have $m + n = 0$, which implies that $\nabla$ is a semi-symmetric metric connection. Under this condition, equations \eqref{geo1} and \eqref{geo2}  yield $\kappa = -n\langle \W, N \rangle = m\langle \W, N \rangle$.  
\end{enumerate}
Conversely,   if condition (i) or   (ii) holds, then the left-hand sides of \eqref{geo1} and \eqref{geo2} vanish identically, ensuring that $\kappa_g = 0$. This completes the proof.
\end{proof}

  In the case of a metric connection, a direct example is the following.

\begin{example}
Let $\nabla$ be a metric cl-connection on $\r^2$ determined by $(\W,m,n)$, where $\W = \partial_x = (1,0)$, $m(x,y) =-\frac1x$, $x\neq0$, and $n=-m$. Consider the circle $\gamma(s)=2(\cos(\frac{s}{2}),\sin(\frac{s}{2}))$. Since $N(s)=-(\cos(\frac{s}{2}),\sin(\frac{s}{2}))$ and $\kappa=1/2$, it follows that $\kappa(s)=(m\circ \gamma)(s)\langle \W,N(s)\rangle$. Therefore, by item (ii) of Theorem \ref{geo-cl}, this circle is a geodesic of $\r^2$ with respect to $\nabla$.
\end{example}

Next, we establish the existence and uniqueness theorem  for curves whose geodesic curvature is prescribed by a generic smooth function $k: I \to \r^+$.  

\begin{theorem}\label{t42} Let  $k: I \to \r^+$ be a smooth function. Let $\nabla$ be a cl-connection on $\mathbb{R}^{2}$ determined by $(\W,m,n)$. If $\gamma$ is a smooth curve parametrized by arc length whose  geodesic curvature is $ k(s) > 0$, then its Frenet curvature $\kappa$ satisfies 
\begin{equation}\label{eq}
2\kappa^2 - 2\langle \W, N \rangle(m - n)\kappa + (m^2 + n^2)\langle \W, N \rangle^2 - k(s)^2 = 0.
\end{equation}
In particular, $\kappa$ can be  expressed as
\begin{equation}\label{eq2}
\kappa = \frac{\langle \W, N \rangle(m - n) \pm \sqrt{2k(s)^2 - \langle \W, N \rangle^2(m + n)^2}}{2},
\end{equation}
provided that  
\begin{equation}\label{ed}
2k(s)^2 \ge \langle \W, N \rangle^2(m + n)^2.
\end{equation}
\end{theorem}

\begin{proof}
The condition $\kappa_g(s) = k(s)$ is equivalent to
\begin{align*}
k(s)^2 &= (\kappa + n\langle \W, N \rangle)^2 + (m\langle \W, N \rangle - \kappa)^2 \\
&= \kappa^2 + 2n\langle \W, N \rangle\kappa + n^2\langle \W, N \rangle^2 + m^2\langle \W, N \rangle^2 - 2m\langle \W, N \rangle\kappa + \kappa^2.
\end{align*}
This identity reduces   to \eqref{eq}. This is a quadratic equation in $\kappa$ of the standard form $A\kappa^2 + B\kappa + C = 0$, where the coefficients are given by
$$ A = 2, \quad B = -2\langle \W, N \rangle(m - n), \quad C = (m^2 + n^2)\langle \W, N \rangle^2 - k(s)^2. $$
The   discriminant $\Delta=B^2-4AC$ is given by
\begin{align*}
\frac{\Delta}{4} &= \left[ \langle \W, N \rangle(m - n) \right]^2 - 2 \left[ (m^2 + n^2)\langle \W, N \rangle^2 - k(s)^2 \right] \\ 
&= 2k(s)^2 - \langle \W, N \rangle^2(m + n)^2.
\end{align*}
This yields \eqref{eq2}. The non-negativity of this   discriminant is equivalent to \eqref{ed}, completing the proof.
\end{proof}

 \begin{remark}\label{r43}
Equation \eqref{eq2} can be  interpreted as a non-linear first-order ordinary differential equation for the angle function $\theta(s)$ of the curve. Indeed,  let $\theta(s)$ denote the angle that the unit tangent vector $T(s)$ makes with the positive $x$-axis. Then $T(s) = (\cos\theta(s), \sin\theta(s))$ and $N(s) = (-\sin\theta(s), \cos\theta(s))$. If   $\W = (\cos\theta_0, \sin\theta_0)$, then  
$$ \langle \W, T(s) \rangle = \cos(\theta(s) - \theta_0) \quad \text{and} \quad \langle \W, N(s) \rangle = \sin(\theta(s) - \theta_0). $$
Since   $\kappa = \theta'(s)$,  equation \eqref{eq2} is equivalent to
$$\theta'(s) = \frac{(m - n)\sin(\theta(s) - \theta_0) \pm \sqrt{2k(s)^2 - (m + n)^2\sin^2(\theta(s) - \theta_0)}}{2},$$
where the connection functions $m$ and $n$ are evaluated along the   trace $\gamma(s)$.  
\end{remark}

We now establish the existence of curves with a prescribed geodesic curvature. This result is implicitly contained in 
Theorem \ref{t42} because once $\kappa$ is determined from \eqref{eq} (or \eqref{eq2}), the existence of the 
curve $\gamma$ with Frenet curvature $\kappa$ is guaranteed by the fundamental theorem of curve theory.

\begin{theorem} \label{t44}
Let $k: I \to \mathbb{R}^+$ be a smooth  function with $0\in I$, $p_0\in\r^2$ and $\theta_0^*\in\r$. 
 Let $\nabla$ be a cl-connection on $\mathbb{R}^{2}$ determined by $(\W,m,n)$, where $\W = (\cos\theta_0, \sin\theta_0)$. Assume that at $s=0$, the   function $k$ satisfies the strict inequality
\begin{equation}\label{et}
2k(0)^2 > \sin^2(\theta_0^* - \theta_0)\left( m(p_0) + n(p_0) \right)^2.
\end{equation}
Then,   there exists a unique local smooth curve $\gamma: I \rightarrow \mathbb{R}^2$ parametrized by arc length such that $\gamma(0) = p_0$, $\theta(0) = \theta_0^*$, and the geodesic curvature of $\gamma$ with respect to $\nabla$ is identically $k(s)$.
\end{theorem}

\begin{proof}
Let $\theta(s)$ be the angle function that the unit tangent vector $T(s)$ makes with the positive $x$-axis. As established in Remark~\ref{r43}, we know that $\langle \W, N(s) \rangle = \sin(\theta(s) - \theta_0)$. Moreover,  $\kappa(s) = \theta'(s)$.
 By substituting this  into   \eqref{eq2}, the condition $\kappa_g(s) = k(s)$ reduces to the following ODE for   $\theta(s)$:

\begin{equation}\label{es0}
\theta'(s) =  \frac{(m - n)\sin(\theta - \theta_0) \pm \sqrt{2k(s)^2 - (m + n)^2\sin^2(\theta - \theta_0)}}{2}, 
\end{equation}
where the connection functions $m=m(x,y)$ and $n=n(x,y)$ are evaluated along the   trace $\gamma(s) = (x(s), y(s))$. Moreover, the coordinate functions satisfy  
\begin{align}
x'(s) &= \cos\theta(s), \label{es1} \\
y'(s) &= \sin\theta(s). \label{es2}
\end{align}
Equations \eqref{es0}, \eqref{es1}, and \eqref{es2} define a first-order autonomous system of non-linear ordinary differential equations for the  vector $(x(s), y(s), \theta(s))$ with the  initial data $x(0) = x_0$, $y(0) = y_0$, and $\theta(0) = \theta_0^*$. By hypothesis, the   condition \eqref{et} holds at $s=0$, ensuring that the argument within the radical is  positive at the origin. By the standard existence and uniqueness theorems for  ODEs,     there exists a unique local solution $(x(s), y(s), \theta(s))$ defined on a small open interval $I$ centered at $s=0$. The curve  $\gamma(s) = (x(s), y(s))$ is parametrized by arc length and   satisfies $\kappa_g(s) = k(s)$ by construction. This completes the proof. 
\end{proof}

As an immediate consequence, we can consider the case where the prescribed geodesic curvature $\kappa_g$ is a constant.

\begin{corollary} \label{constgeo}
Let $\nabla$ be a cl-connection on $\mathbb{R}^{2}$ determined by $(\W,m,n)$. If a smooth curve $\gamma$ parametrized by arc length has a constant geodesic curvature $\kappa_g = c > 0$, satisfying the restriction $2c^2 \ge \langle \W, N \rangle^2(m + n)^2$, then its Frenet curvature  $\kappa$ satisfies  
$$
\kappa = \frac{\langle \W, N \rangle(m - n) \pm \sqrt{2c^2 - \langle \W, N \rangle^2(m + n)^2}}{2}.
$$
In the particular case that the cl-connection $\nabla$ coincides with the Levi-Civita connection $\nabla^0$ of the Euclidean plane (i.e., $m = n = 0$), a curve $\gamma$ has constant geodesic curvature $\kappa_g = c > 0$ if and only if  $\kappa = \pm c$.  
\end{corollary}

We give two examples of curves with constant geodesic curvature $\kappa_g=1$. The examples depend on the choice of the connection functions $m$ and $n$.

\begin{example} Assume $\W=(1,0)$ and $\kappa_g=1$ identically along the curve $\gamma$.
\begin{enumerate}
\item[(i)] Let   $m = n = \frac{1}{2}$.  Under these choices,    the restriction $2c^2 \ge \langle \W, N \rangle^2(m+n)^2$ reduces to $2 \ge \langle \W, N \rangle^2$, which is trivially  satisfied since $\W$ and $N$ are unit vector fields. Substituting these parameters into Corollary \ref{constgeo}, the   Frenet curvature $\kappa$   must satisfy
$$\kappa = \pm \frac{\sqrt{2 - \langle \W, N \rangle^2}}{2},$$
or equivalently,  
$$\theta'(s) = \pm \frac{\sqrt{2 - \sin^2\theta(s)}}{2}.$$
By separating variables, this ODE  can be rewritten in terms of elliptic integrals as 
$$\int \frac{d\theta}{\sqrt{1 - \frac{1}{2}\sin^2\theta}} = \pm \frac{1}{\sqrt{2}}\int\, ds.$$
The left-hand side is the incomplete elliptic integral of the first kind, $F(\theta, k)$, with modulus $k = \frac{\sqrt{2}}{2}$. Therefore, by inverting this relation via the Jacobi amplitude function $\text{am}(u, k)$, we obtain 
$$\theta(s) = \text{am}\left(\pm \frac{s}{\sqrt{2}} + c, \, \frac{\sqrt{2}}{2}\right),$$
where $c \in\r$ is an integration constant determined by the initial orientation of the curve. 
\item[(ii)] Let  $m = 1$ and $n = -1$.  Then   the  restriction $2c^2 \ge \langle \W, N \rangle^2(m+n)^2$ reduces to $2 \ge 0$, which is always satisfied. Since $\langle \W, N \rangle = -\sin\theta(s)$,  and $\kappa = \theta'(s)$, then 
$$\theta'(s) = -\sin\theta(s) \pm \frac{\sqrt{2}}{2}.$$
Separation of variables leads to 
$$\theta(s) = 2 \arctan\left( \sqrt{2} \pm \tan\left( \frac{\sqrt{2}}{4}(s + c) \right) \right),$$
where $c \in \mathbb{R}$ is an integration constant. Next, we compute the parametrization of the curve $\gamma$. For simplicity, we assume that $c=0$ and choose the positive sign in the expression for $\theta$. Using the identities
$$
\cos(2\arctan v)=\frac{1-v^2}{1+v^2}, \quad \sin(2\arctan v)=\frac{2v}{1+v^2},
$$
we have
\begin{equation}\label{kg=1}
\gamma(s)
=
\left(
\int_0^s
\frac{1-v(u)^2}
     {1+v(u)^2}
\,du,
\int_0^s
\frac{2v(u)}
     {1+v(u)^2}
\,du
\right), \quad v(u):=\sqrt{2}+\tan(\frac{\sqrt{2}}{4} u).
\end{equation}
The resulting curve is illustrated in Fig. \ref{fig2}.
\end{enumerate}
\end{example}

\begin{figure}[ht]
\centering
\includegraphics[width=0.3\textwidth]{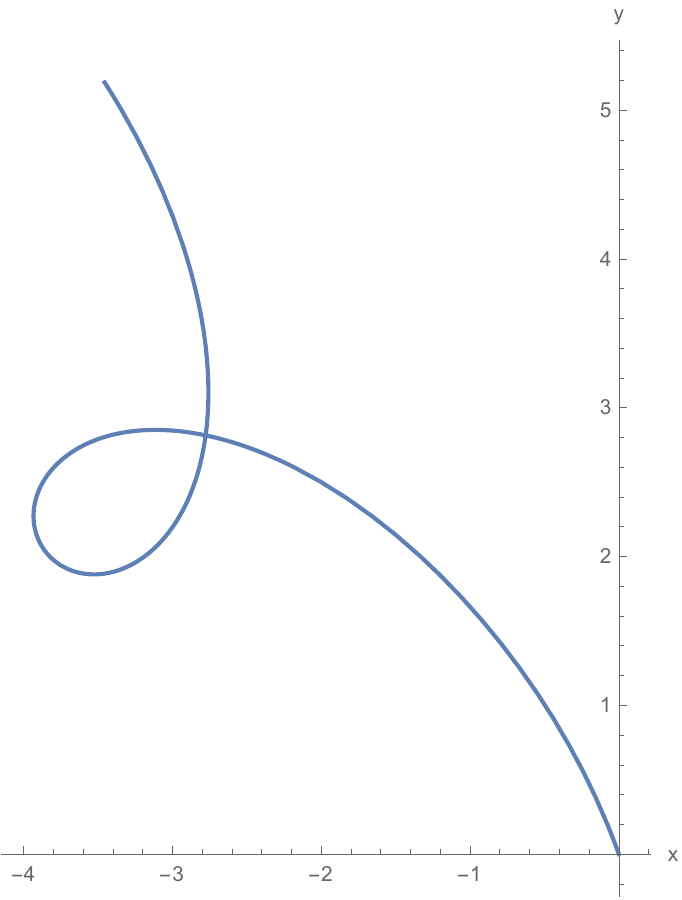}
\caption{A curve with constant geodesic curvature $\kappa_g(s)=1$, defined by \eqref{kg=1} for $s\in[0,10]$.}
\label{fig2}
\end{figure}

We finish the section by giving a geometric interpretation of curvatures via relative angles. In the classical Euclidean planar geometry, the Frenet curvature $\kappa$ is the derivative of the tangent angle with respect to a fixed coordinate axis of $\mathbb{R}^2$. In this setting, we analyze  how the unit tangent vector $T$ varies relative to the reference field $\W$.

Let $\W = (\cos\theta_0, \sin\theta_0)$ be the prescribed constant unit vector field defining the cl-connection $\nabla$. Let $\gamma: I \rightarrow \mathbb{R}^2$ be a smooth curve parametrized by arc length, and let $\theta(s)$ denote the angle function that the unit tangent vector $T(s)$ makes with the positive $x$-axis. Then we have:
$$ T(s) = (\cos\theta(s), \sin\theta(s)), \quad N(s) = (-\sin\theta(s), \cos\theta(s)). $$
We introduce the \textit{relative angle function} $\psi(s)$, defined as the angle from the constant field $\W$ to the unit tangent vector field $T(s)$ along $\gamma$, which is given by
\begin{equation}\label{psi}
\psi(s) = \theta(s) - \theta_0.
\end{equation}

The following theorem establishes a relationship between the  curvatures and the relative angle function.

\begin{theorem} \label{t46}
Let $\nabla$ be a cl-connection on $\mathbb{R}^2$ determined by $(\W,m,n)$. Let $\gamma$ be a smooth curve parametrized by arc length, and let $\psi(s)$ be its relative angle function with respect to $\W$. Then:
\begin{align}
    \kappa_t(s) &= (m+n)\cos\psi(s), \label{k1} \\
    \kappa_g(s) &= \sqrt{ (\psi'(s) - n\sin\psi(s))^2 + (\psi'(s) + m\sin\psi(s))^2 }, \label{k2}
\end{align}
where $m=m\circ\gamma$ and $n=n\circ\gamma$.  
\end{theorem}

\begin{proof}
Since $\W = (\cos\theta_0, \sin\theta_0)$, we have $\langle \W, T(s) \rangle  = \cos\psi(s)$ and $\langle \W, N(s) \rangle  = -\sin\psi(s)$. From \eqref{d1}, the relation   $\kappa_t = (m+n)\langle \W, T \rangle$ yields \eqref{k1}.

To prove \eqref{k2}, we differentiate the relative angle \eqref{psi} with respect to the arc length parameter $s$ to obtain  $ \psi'(s) = \theta'(s) = \kappa(s)$.  Substituting $\kappa = \psi'$ and $\langle \W, N \rangle = -\sin\psi(s)$ into \eqref{d2},  we arrive at  
\begin{equation*}
\begin{split}
\kappa_g^2 &= (\kappa + n\langle \W, N \rangle)^2 + (m\langle \W, N \rangle - \kappa)^2\\
&= (\psi' - n\sin\psi)^2 + (\psi' + m\sin\psi)^2.
\end{split}
\end{equation*} 
\end{proof}

\begin{remark} 
This result shows  that the tangential curvature $\kappa_t$ acts as a directional projection of the connection due to the factor $\cos\psi$. It vanishes   when the curve moves   orthogonally to the    vector field $\W$ (i.e., $\psi = \pm \pi/2$), regardless of how sharply the curve is turning. 

\end{remark}

 
\section{Modified $\nabla$-spirals}\label{s5}

In the classical differential geometry of plane curves, spirals are curves  characterized by a strict monotonic relation between
 their   curvature and their arc length parameter. Two of the most important examples are  clothoids (also known as the 
 Euler or Cornu spirals), where the curvature varies linearly as $\kappa(s) =  s$, and     logarithmic spirals, where 
 $\kappa(s)=1/s$. The parametrization of   clothoids  involves the 
 Fresnel integrals, whereas for   logarithmics spirals, the parametrizations are   given in terms of   logarithmic  and trigonometric functions.     

In our setting, we  generalize both types of curves to obtain   two new classes of   curves: the \textit{ 
$\nabla$-clothoids} and the \textit{ logarithmic $\nabla$-spirals}. Let $\nabla$ be a 
non-metric cl-connection determined by $(\W,m,n)$, where $\W = (\cos\theta_0, \sin\theta_0)$. To simplify the discussion, we will assume in this section  that the connection functions satisfy the constancy   condition 
$$m(x,y) + n(x,y) = \lambda$$
 along the curve, where $\lambda > 0$ is a constant. Recall that   $R_{\theta_0}$ represents the   planar rotation matrix about the origin by the angle $\theta_0$.

We first study the analogue of the classical clothoid by investigating curves   whose tangential curvature is $\kappa_t(s) = s$ (see Fig. \ref{fig3}).

 \begin{theorem}\label{t51}
Let $\nabla$ be a cl-connection determined by $(\W,m,n)$, where $\W = (\cos\theta_0, \sin\theta_0)$ and $m+n = \lambda > 0$ along the curve. Let 
$\gamma: I \rightarrow \r^2$ be a smooth curve parametrized by arc length with $\kappa_t(s) = s$ for $s \in (- \lambda, \lambda)$. Then, up to a rigid translation, $\gamma(s)$ is   given by 
\begin{equation}\label{ec}
\gamma(s) = R_{\theta_0} \circ \left( \pm \frac{s^2}{2\lambda}, \, \pm \left[ \frac{s}{2}\sqrt{1 - \frac{s^2}{\lambda^2}} + \frac{\lambda}{2}\arcsin\left(\frac{s}{\lambda}\right) \right] \right).
\end{equation}
\end{theorem}

\begin{proof}
The result follows as a direct application of Theorem~\ref{t32}. Since  
$\kappa_t(s) = s$, and   $m+n = \lambda > 0$,     the auxiliary function is
$$ f(s) = \frac{\kappa_t(s)}{m+n} = \frac{s}{\lambda}. $$
Since $\W$ and $T$ are unit vector fields, the condition $|f(s)| \le 1$   imposes the domain restriction $|s| \le \lambda$. Substituting $f(t) = \frac{t}{\lambda}$   into   \eqref{e4}, we evaluate the two corresponding 
integrals. The first integral   yields 
$$ \int^s f(t) \, dt = \int^s \frac{t}{\lambda} \, dt = \frac{s^2}{2\lambda}. $$
The second integral is a standard trigonometric integral whose value is  
$$ \int^s \sqrt{1 - f(t)^2} \, dt = \int^s \sqrt{1 - \frac{t^2}{\lambda^2}} \, dt = \frac{s}{2}\sqrt{1 - \frac{s^2}{\lambda^2}} + \frac{\lambda}{2}\arcsin\left(\frac{s}{\lambda}\right). $$
Inserting these   integrals   into   \eqref{e4} yields  \eqref{ec}, completing the proof.
\end{proof}

\begin{remark} 
We compare the geometric behavior of clothoids and $\nabla$-clothoids.  A standard clothoid spirals infinitely toward an asymptotic center point as 
$s \to \infty$, because its curvature grows unbounded ($\kappa \to \infty$). For the modified $\nabla$-clothoid under 
a cl-connection, the curve cannot extend infinitely in terms of arc length because its parameter is bounded by    $|s| < \lambda$.
  As $s \to \pm \lambda$, equation \eqref{ec} shows that  $\kappa \to \infty$. This means the curve ends at a point   where it turns infinitely fast.

\end{remark}

\begin{remark} 
While Theorem~\ref{t51} determines the trajectory uniquely by prescribing its tangential curvature $\kappa_t(s) = s$, its geodesic curvature $\kappa_g$ is not uniquely constrained. The tangential curvature depends only on the sum $m+n = \lambda$. In contrast, the geodesic curvature evaluates in separate terms, namely $m\langle\W,N\rangle$ and $n\langle \W,N\rangle$. In the present case,  $\langle \W, N \rangle = \pm\sqrt{1 - s^2/\lambda^2}$ and  $\kappa\langle \W, N \rangle = f'(s) = 1/\lambda$. This yields   $\kappa^2 = \frac{1}{\lambda^2 - s^2}$. Thus 
$$ \kappa_g^2(s) = \frac{2}{\lambda^2 - s^2} - \frac{2(m-n)}{\lambda} + (m^2+n^2)\left(1 - \frac{s^2}{\lambda^2}\right). $$
This     formula indicates that  the  geodesic curvature   depends   on how  $\lambda$ is   distributed between $m(x,y)$ and $n(x,y)$.
\end{remark}

Next, we     generalize the logarithmic  spirals by prescribing a tangential curvature   $\kappa_t(s) = 1/s$ (see Fig. \ref{fig3}).

\begin{theorem}
Let $\nabla$ be a cl-connection on $\r^2$ with $m+n = \lambda > 0$ along the curve. Let $\gamma: I \rightarrow \r^2$ be a smooth curve parametrized by 
arc length with  $\kappa_t(s) = 1/s$ for $s > 0$. Then, the domain of definition for the curve is restricted to $s \ge 1/\lambda$, and up to rigid translations, $\gamma(s)$ is   given by 
\begin{equation}\label{e32}
\gamma(s) = R_{\theta_0} \left( \pm \frac{\log(s)}{\lambda}, \, \pm \frac{1}{\lambda}\left[ \sqrt{\lambda^2 s^2 - 1} - \arccos\left(\frac{1}{\lambda s}\right) \right] \right).
\end{equation}
\end{theorem}

\begin{proof}
The proof relies on a direct application of Theorem~\ref{t32}. Given the prescribed tangential curvature $\kappa_t(s) = 1/s$ 
and the  condition $m+n = \lambda > 0$, we obtain  
$$f(s) = \frac{\kappa_t(s)}{m+n} = \frac{1}{\lambda s}.$$
Recall that we require $|f(s)| \le 1$. Assuming $s > 0$, this yields  $s \ge 1/\lambda$. To determine the  parametrization, we substitute $f(t) = \frac{1}{\lambda t}$ into   \eqref{e4} and evaluate both integrals. The first integral is  
$$\int^s f(t) \, dt = \int^s \frac{1}{\lambda t} \, dt = \frac{\log(s)}{\lambda}.$$
For the second integral, a simple integration yields 
$$\int^s \sqrt{1 - f(t)^2} \, dt = \frac{1}{\lambda} \left[ \sqrt{\lambda^2 s^2 - 1} - \arccos\left(\frac{1}{\lambda s}\right) \right].$$
Inserting these   integrals back into   \eqref{e4} produces the parametrization \eqref{e32}, thereby concluding the proof.
\end{proof}
\begin{figure}[ht]
\centering
\includegraphics[width=0.2\textwidth]{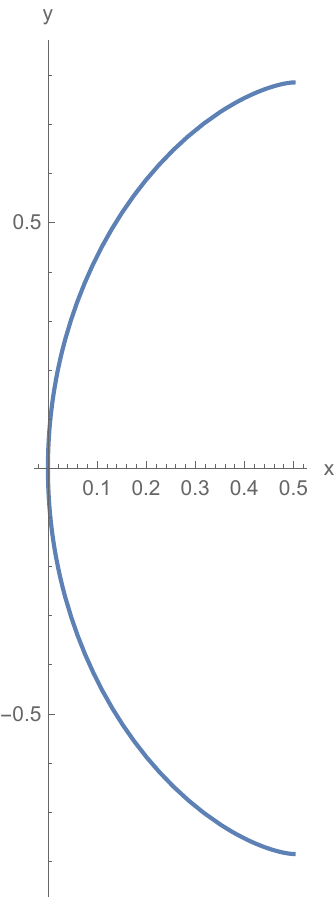} \quad \includegraphics[width=0.25\textwidth]{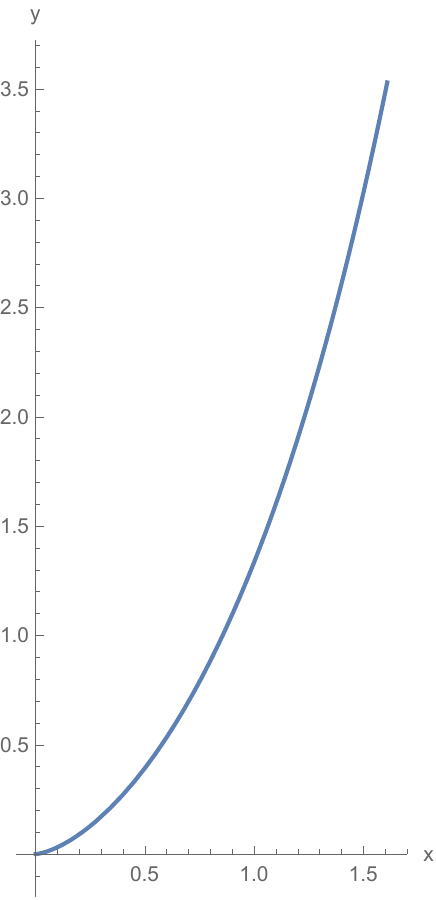}
\caption{A $\nabla$-clothoid (left) given by \eqref{ec} and a logarithmic $\nabla$-spiral (right) given by \eqref{e32}, both corresponding to $\lambda=1$ and the positive sign.}
\label{fig3}
\end{figure}
\section*{Acknowledgements}
Rafael L\'opez has been partially supported by MINECO/MICINN /FEDER grant no. PID2023-150727NB-I00,  and by the ``Mar\'{\i}a de Maeztu'' Excellence Unit IMAG, reference CEX2020-001105- M, funded by MCINN/ AEI/10.13039/ 501100011033/ CEX2020-001105-M.



\begin{thebibliography}{99}




\bibitem{ac} N. S. Agashe, M. R. Chafle, On submanifolds of a Riemannian manifold with a semi-symmetric non-metric connection. Tensor 55 (1994), 120--130.


\bibitem{fs} A. Friedmann, J. A. Schouten, \"{U}ber die Geometrie der halbsymmetrischen \"{U}bertragungen. Math. Z. 21 (1924), 211--223.

 
\bibitem{gu} \c{S}. G\"{u}ven\c{c}, Constructions of Frenet curves with respect to semi-symmetric metric connection. arXiv:2408.05880v1 [math.DG] 2024.
 
\bibitem{hay} H. Hayden, Subspaces of a space with torsion. Proc. London Math. Soc. 34 (1932), 27--50.

\bibitem{im} T. Imai, Hypersurfaces of a Riemannian manifold with semi-symmetric metric connection. Tensor (N. S.) 23 (1972), 300--306. 

\bibitem{lyl} C. W. Lee, D. W. Yoon, J. W. Lee, Optimal inequalities for the Casorati curvatures of submanifolds of real space forms endowed with semi-symmetric metric connections. J. Inequal. Appl. 2014, 327 (2014).



 
 \bibitem{mc1} A. Mihai, C. \"{O}zg\"{u}r, Chen inequalities for submanifolds of real space forms with a semi-symmetric metric connection. Taiwan. J. Math. 4 (2010), 1465--1477.

%

\bibitem{na} Z. Nakao, Submanifolds of a Riemannian manifold with semisymmetric metric connections. Proc. Amer. Math. Soc. 54 (1976), 261--266.

 

\bibitem{pre} A. Pressley, Elementary Differential Geometry, 2nd ed., Springer, Berlin/Heidelberg, 2001.

\bibitem{trip} M. M. Tripathi, A new connection in a Riemannian manifolds. Int. Electron. J. Geom. 1(1) (2008), 15-24.

\bibitem{ya} K. Yano, On semi symmetric metric connection. Rev. Roum. Math. Pures Appl. 15 (1970), 1579--1591.
 


   
 

\end{thebibliography}
\end{document}